\documentclass{article}
\usepackage[left=30truemm, right=30truemm]{geometry}
\usepackage{amsmath,amssymb,amsthm}
\usepackage{bm}
\usepackage{mathtools}
\usepackage{xcolor}

\usepackage[shortlabels]{enumitem}
\usepackage{float}

\usepackage{tikz}
\usepackage{caption}
\usepackage{subcaption}
\usetikzlibrary{arrows.meta,calc,decorations.markings,math,arrows.meta}

\usepackage[normalem]{ulem}

\newtheorem{theorem}{Theorem}[]
\newtheorem{proposition}[theorem]{Proposition}

\newtheorem{lemma}[theorem]{Lemma}

\newtheorem{conjecture}[theorem]{Conjecture}

\newtheorem{question}[]{Question}
\newtheorem{problem}[question]{Problem}
\newtheorem{case}[]{Case}
\newtheorem{subcase}{Subcase}[case]

\usepackage[pdfauthor={derajan},pdfstartview=XYZ,bookmarks=true,
colorlinks=true,linkcolor=blue,urlcolor=magenta,citecolor=blue,pdftex,bookmarks=true,linktocpage=true,hyperindex=true]{hyperref}

\numberwithin{equation}{section}

\title{Proper conflict-free degree-choosability of outerplanar graphs}
\author{
    Masaki Kashima\thanks{Faculty of Science and Technology, Keio University, Yokohama, Japan. Email: masaki.kashima10@gmail.com} \quad  
	Riste \v{S}krekovski\thanks{Faculty of Mathematics and Physics, University of Ljubljana, Faculty of Information Studies in Novo Mesto, and Rudolfovo - Science and Technology Centre Novo Mesto, Slovenia. Email:skrekovski@gmail.com} \quad 
	Rongxing Xu\thanks{School of Mathematical and Science, Zhejiang Normal University, Jinhua, China. Email:xurongxing@zjnu.edu.cn}}

\begin{document}

\maketitle

\begin{abstract}
    A proper coloring $\phi$ of $G$ is called a proper conflict-free coloring of $G$ if for every non-isolated vertex $v$ of $G$, there is a color $c$ such that $|\phi^{-1}(c)\cap N_G(v)|=1$.
    As an analogy to degree-choosability of graphs, the authors recently, in a previous paper, introduced the notion of proper conflict-free $({\rm degree}+k)$-choosability of graphs.
    For a non-negative integer $k$, a graph $G$ is proper conflict-free $({\rm degree}+k)$-choosable if for any list assignment $L$ of $G$ with $|L(v)|\geq d_G(v)+k$ for every vertex $v\in V(G)$, $G$ admits a proper conflict-free coloring $\phi$ such that $\phi(v)\in L(v)$ for every vertex $v\in V(G)$.
    In this paper, we show that every connected outerplanar graph other than the $5$-cycle is proper conflict-free $({\rm degree}+2)$-choosable.
    This bound is tight in the sense that there are infinitely many connected outerplanar graphs that are not proper conflict-free $({\rm degree}+1)$-choosable.
    We conclude the paper with two questions for further work.
\end{abstract}

\section{Introduction}\label{sec:intro}

Throughout the paper, we only consider simple, finite, and undirected graphs.
Let $\mathbb{N}$ be the set of positive integers.
For a positive integer $k$, let $[k]$ denote the set of integers $\{1,2,\dots , k\}$.

For a graph $G$, a mapping $\phi$ from $V(G)$ to $\mathbb{N}$ is called a \emph{proper coloring of $G$} if $\phi(u)\neq \phi(v)$ for every edge $uv\in E(G)$.
A proper coloring of a graph $G$ in which every vertex of $G$ maps to an integer in $[k]$ is called a proper $k$-coloring of $G$.

Recently, Fabrici, Lu\v{z}ar, Rindo\v{s}ov\'{a}, and Sot\'{a}k~\cite{FLRS2023} introduced a new variation of coloring named proper conflict-free coloring of graphs.
For a graph $G$, a mapping $\phi$ from $V(G)$ to $\mathbb{N}$ is called a \emph{proper conflict-free coloring of $G$} if $\phi$ is a proper coloring of $G$ and every non-isolated vertex $v\in V(G)$ has a color $c$ such that $|\phi^{-1}(c)\cap N_G(v)|=1$, where $N_G(v)$ is the (open) neighborhood of $v$.
A proper conflict-free coloring of a graph $G$ such that every vertex of $G$ maps to an integer in $[k]$ is called a \emph{proper conflict-free $k$-coloring of $G$}.
For a (partial) coloring $\phi$ of $G$ and a vertex $v\in V(G)$, let $\mathcal{U}_{\phi}(v,G)$ denote the set of colors that appear exactly once in the neighborhood of $v$.
Using this notation, a proper conflict-free coloring $\phi$ of $G$ is a proper coloring $\phi$ of $G$ such that $\mathcal{U}_{\phi}(v,G)\neq \emptyset$ for every non-isolated vertex $v\in V(G)$.
The \emph{proper conflict-free chromatic number} of a graph $G$, denoted by $\chi_{\text{pcf}}(G)$, is the least integer $k$ such that $G$ admits a proper conflict-free $k$-coloring.

One major problem in proper conflict-free coloring is the following Brooks-type conjecture, which was posed by Caro, Petru\v{s}evski, and \v{S}krekovski~\cite{CPS2023}.

\begin{conjecture}\label{conj:brooks}
    For every graph $G$ with maximum degree $\Delta\geq 3$, $\chi_{{\rm pcf}}(G)\leq \Delta+1$.
\end{conjecture}

Toward this conjecture, studies on Brooks-type results for proper conflict-free coloring have been established in the literature~\cite{CCKP2025,CLarxiv,KP2024,L2024,LRarxiv,LY2013}.

Motivated by Conjecture~\ref{conj:brooks} and the classical concept of degree-choosability, we introduced the concept of proper conflict-free $({\rm degree}+k)$-choosability of graphs in \cite{KSXarxiv}, where we investigated the proper conflict-free $({\rm degree}+k)$-chooability of graphs with bounded maximum degree.

A list assignment $L$ of a graph $G$ is a mapping from $V(G)$ to the power set of $\mathbb{N}$.
For a mapping $f$ from $V(G)$ to positive integers, a list assignment $L$ of a graph $G$ is called an $f$-list assignment of $G$ if $|L(v)|\geq f(v)$ for every vertex $v\in V(G)$.
In particular, if $f$ is the constant map from $V(G)$ to a positive integer $k$, an $f$-list assignment of $G$ is called a $k$-list assignment of $G$.

For a given graph $G$ and a list assignment $L$ of $G$, a proper conflict-free $L$-coloring of $G$ is a proper conflict-free coloring $\phi$ of $G$ such that $\phi(v)\in L(v)$ for every vertex $v\in V(G)$.
For a non-negative integer $k$, a graph $G$ is \emph{proper conflict-free $({\rm degree}+k)$-choosable} if $G$ admits a proper conflict-free $L$-coloring for any $f$-list assignment of $G$, where $f(v)=d_G(v)+k$ for every vertex $v\in V(G)$.
In \cite{KSXarxiv}, we pose the following conjecture that states a fundamental problem in this concept.

\begin{conjecture}\label{conj:general}
	There exists an absolute constant $k$ such that every graph is proper conflict-free $({\rm degree}+k)$-choosable.
\end{conjecture}

If Conjecture~\ref{conj:general} is true, then it guarantees the existence of an absolute constant $k$ such that every graph $G$ satisfies $\chi_{{\rm pcf}}(G)\leq \Delta(G)+k$, where $\Delta(G)$ is the maximum degree of $G$.

In this paper, we focus on proper conflict-free $({\rm degree}+k)$-choosability of outerplanar graphs.
It is known that the 5-cycle is not proper conflict-free 4-colorable, and hence it is not proper conflict-free $({\rm degree}+2)$-choosable.
Our main result is the following, which shows that the $5$-cycle is the only outerplanar graph that is not proper conflict-free $({\rm degree}+2)$-choosable.

\begin{theorem}\label{thm:outerplanar}
  Every connected outerplanar graph other than the $5$-cycle is proper conflict-free $({\rm degree}+2)$-choosable.
\end{theorem}

The number of colors $({\rm degree}+2)$ in Theorem~\ref{thm:outerplanar} cannot be reduced to $({\rm degree}+1)$.
Indeed, we can construct connected outerplanar graphs that are not proper conflict-free $(\text{degree}+1)$-choosable in the following manner.

\begin{proposition}\label{prop:degree+1}
    For every connected outerplanar graph $H$, there is a connected outerplanar graph $G$ such that $H$ is an induced subgraph of $G$ and $G$ is not proper conflict-free $({\rm degree}+1)$-choosable.
\end{proposition}

\begin{proof}
    Let $H$ be a connected outerplanar graph and let $v_0$ be a vertex of $H$ of degree $d$.
    For each $i\in [d+1]$, let $C^{(i)}=v_0^iv_1^iv_2^iv_3^iv_0^i$ be a cycle of length $4$.
    Let $G$ be a graph obtained from $H$ and $C^{(1)}, C^{(2)}, \dots , C^{(d+1)}$ by identifying $v_0$ and $v_0^1, v_0^2, \dots , v_0^{d+1}$ into a vertex $v$.
    Obviously, $G$ is a connected outerplanar graph that has $H$ as an induced subgraph.
    
    We define a list assignment $L$ of $G$ as follows:
    \begin{itemize}
        \item for each $i\in \{1,2,\dots , d+1\}$, let $L(v_1^i)=L(v_2^i)=L(v_3^i)=\{3i-2, 3i-1, 3i\}$,
        \item $L(v)=\{1,2,\dots , 3d+3\}$, and
        \item for each vertex $u\in V(H)\setminus \{v\}$, let $L(u)=\{1,2,\dots , d_G(u)+1\}$.
    \end{itemize}
    Since $d_G(v)=d_H(v_0)+2(d+1)=d+2(d+1)=3d+2$, $L$ is a list assignment of $G$ such that $|L(u)|\geq d_G(u)+1$ for every vertex $u$ of $G$.
    We show that $G$ is not proper conflict-free $L$-colorable.
    Suppose that $G$ admits a proper conflict-free $L$-coloring $\phi$.
    For each $i\in [d+1]$, since $\mathcal{U}_{\phi}(v_2^i,G)\neq \emptyset$, we have $\{\phi(v_1^i),\phi(v_2^i),\phi(v_3^i)\}=\{3i-2,3i-1,3i\}$.
    In addition, since $\phi(v_j^i)\neq \phi(v)$ and $\mathcal{U}_{\phi}(v_j^i,G)\neq \emptyset$ for $j=1,3$, we have $\phi(v)\notin \{\phi(v_1^i),\phi(v_2^i),\phi(v_3^i)\}=\{3i-2,3i-1,3i\}$.
    This implies that $\phi(v)\notin \bigcup_{i=1}^{d+1}\{3i-2,3i-1,3i\}=\{1,2,\dots , 3d+3\}=L(v)$, a contradiction.
\end{proof}

While the scope of Theorem~\ref{thm:outerplanar} is outerplanar graphs, as we mentioned in our previous paper \cite{KSXarxiv}, we know no other connected graph other than the $5$-cycle that is not proper conflict-free $({\rm degree}+2)$-choosable.
This leads us to the following stronger conjecture originally stated in \cite{KSXarxiv}.

\begin{conjecture}\label{conj:general degree+2}
    Every connected graph other than the $5$-cycle is proper conflict-free $({\rm degree}+2)$-choosable.
\end{conjecture}

The paper is organized as follows. 
In Section~\ref{sec:lemma}, we show some propositions and lemmas used in our proof of Theorem~\ref{thm:outerplanar}.
In Section~\ref{sec:proof}, we give a proof of Theorem~\ref{thm:outerplanar}.
In Section~\ref{sec:conclusion}, we conclude the paper and propose several lines for future work.

\section{Auxiliary results}\label{sec:lemma}

\begin{proposition}\label{prop:cycle}
  Every cycle of length $\ell\neq 5$ is proper conflict-free $({\rm degree}+2)$-choosable.
\end{proposition}

\begin{proof}
    The statement is trivial when $\ell\leq 4$. Assume that $\ell\geq 6$.
    Let $C=v_0v_1\cdots v_{\ell-1}v_0$ be a cycle of length $\ell\geq 6$. 
    Let $L$ be a list assignment of $C$ such that $|L(v_i)|=d_C(v_i)+2=4$ for every vertex $i\in\{0,1,\dots ,\ell-1\}$.
    If $L(v_0)=L(v_1)=\cdots =L(v_{\ell-1})$, then $C$ is proper conflict-free $L$-colorable since the cycle of length $\ell\neq 5$ is proper conflict-free $4$-colorable, as shown in Caro, Petru\v{s}evski, and \v{S}krekovski~\cite{CPS2023}.
    Thus, without loss of generality, we assume that $L(v_0)\neq L(v_{\ell-1})$.
    
    Let $\alpha$ be a color in $L(v_0)\setminus L(v_{\ell-1})$.
    We set $\phi(v_0)=\alpha$ and set $\phi(v_1)\in L(v_1)\setminus\{\alpha\}$.
    For $i=2,3,\dots , \ell-3$, we set $\phi(v_i)\in L(v_i)\setminus\{\phi(v_{i-2}),\phi(v_{i-1})\}$.
    Let $\phi(v_{\ell-2})\in L(v_{\ell-2})\setminus\{\phi(v_{\ell-4}),\phi(v_{\ell-3}),\alpha\}$ and let $\phi(v_{\ell-1})\in L(v_{\ell-1})\setminus\{\phi(v_{\ell-3}),\phi(v_{\ell-2}), \phi(v_1)\}$.
    Since $\alpha\notin L(v_{\ell-1})$, we have $\phi(v_{\ell-1})\neq \phi(v_0)$.
    By the choice of colors, it is easy to verify that $\phi$ is a proper conflict-free $L$-coloring of $C$.
\end{proof}

Note that the latter argument of the proof of Proposition~\ref{prop:cycle} works even if $\ell=5$.
This implies the following claim.

\begin{proposition}\label{prop:c5}
  Let $C=v_0v_1v_2v_3v_4v_0$ be a $5$-cycle.
  Suppose that $L$ is a $4$-list assignment of $C$.
  Then $C$ is not proper conflict-free $L$-colorable if and only if $L(v_0)=L(v_1)=L(v_2)=L(v_3)=L(v_4)$ and $|L(v_0)|=4$.
\end{proposition}

In our proof of Theorem~\ref{thm:outerplanar}, we use an unavoidable set of $2$-connected outerplanar graphs.
Let $G$ be a $2$-connected outerplanar graph which is not a cycle.
Let us define:

\begin{itemize}
  \item An \emph{ear} $H$ of $G$ is a cycle $u_1u_2\cdots u_ru_1$ $(r\geq 3)$ such that $d_G(u_i) =2$ for $i\in \{2,3,\dots , r-1\}$. (Figure~\ref{fig:ear}) 
  The edge $u_1u_r$ is the \emph{root edge} of $H$. 
  We say $H$ is \emph{good for a vertex $x$ of $G$} if $d_G(u_i)=3$ and $x\notin V(H)\setminus\{u_{r+1-i}\}$ for some $i\in\{1,r\}$.
  \item An \emph{ear-chain} $H$ is a subgraph of $G$ induced by $\bigcup_{i=1}^{s-1}V(H_i)$ $(s\geq 3)$ where each $H_i$ is an ear of $G$ such that the root edges of the ears form a path $v_1v_2 \cdots v_s$
  (the root edge of $H_i$ is $v_iv_{i+1}$) together with the edge $v_1v_s$ of $G$ and $d_G(v_i) = 4$ for $i\in\{2,3,\dots, s-1\}$. (Figure~\ref{fig:ear chain})
  The edge $v_1v_s$ is the \emph{root edge} of $H$.
  We say $H$ is \emph{good for a vertex $x$ of $G$} if $x\notin V(H)\setminus\{v_1,v_s\}$.
\end{itemize}

\begin{figure}
  \centering
  \begin{minipage}{0.45\columnwidth}
    \centering
    \begin{tikzpicture}[roundnode/.style={circle, draw=black,fill=black, minimum size=1.5mm, inner sep=0pt},
	smallnode/.style={circle, draw=black,fill=black, minimum size=1.2mm, inner sep=0pt},]
			
		\draw (2,0) arc [start angle=0, end angle=180, radius=2];

		\node [roundnode] (u1) at (10:2){};
		\node [roundnode] (u2) at (30:2){};
		\node [roundnode] (u3) at (50:2){};
		\node [roundnode] (u4) at (70:2){};
		\node [roundnode] (u5) at (90:2){};
		\node [roundnode] (u6) at (110:2){};
		\node [roundnode] (u7) at (130:2){};
		\node [roundnode] (u8) at (150:2){};
		\node [roundnode] (u9) at (170:2){};

		\draw (u1)--(u9);

		\node at (10:2.3){$u_1$};
		\node at (30:2.3){$u_2$};
		\node at (170:2.3){$u_r$};
    \end{tikzpicture}
    \subcaption{An ear}
    \label{fig:ear}
  \end{minipage}
  \begin{minipage}{0.45\columnwidth}
    \centering
    \begin{tikzpicture}[roundnode/.style={circle, draw=black,fill=black, minimum size=1.5mm, inner sep=0pt},
		smallnode/.style={circle, draw=black,fill=black, minimum size=1.2mm, inner sep=0pt},]
			
		\draw (2,0) arc [start angle=0, end angle=180, radius=2];

		\node [roundnode] (v1) at (10:2){};
		\node [roundnode] (v2) at (50:2){};
		\node [roundnode] (v3) at (110:2){};
		\node [roundnode] (v4) at (170:2){};

		\node [smallnode] at (30:2){};
		\node [smallnode] at (62:2){};
		\node [smallnode] at (74:2){};
		\node [smallnode] at (86:2){};
		\node [smallnode] at (98:2){};
		\node [smallnode] at (130:2){};
		\node [smallnode] at (150:2){};

		\draw (v1)--(v2)--(v3)--(v4)--(v1);

		\node at (10:2.3){$v_1$};
		\node at (50:2.3){$v_2$};
		\node at (170:2.3){$v_s$};
        \node at (30:2.5){$H_1$};
        \node at (140:2.5){$H_{s-1}$};
    \end{tikzpicture}
    \subcaption{An ear chain}
    \label{fig:ear chain}
    \end{minipage}
	\caption{Unavoidable structures of 2-connected outerplanar graphs.}
	\label{fig:unavoidable}
\end{figure}
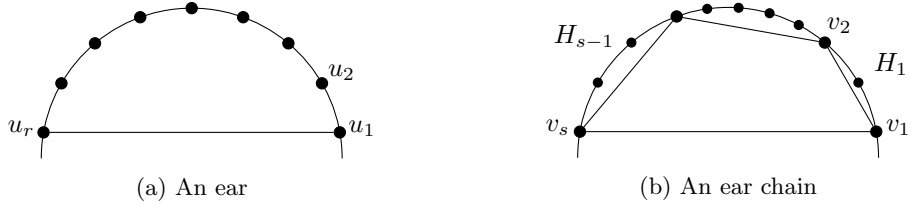

Then, the following holds.

\begin{lemma}\label{lem:outerplanar unavoidable}
  Let $G$ be a $2$-connected outerplanar graph of order at least $4$, and let $x$ be a vertex of $G$.
  If $G$ is not a cycle, then $G$ contains an ear or an ear-chain that is good for $x$.
\end{lemma}

\begin{proof}
  Let $G$ be a $2$-connected outerplanar graph which is not a cycle.
  We fix a vertex $x\in V(G)$.
  Let $F$ be the outercycle of $G$.
  Let $E_1$ be the set of edges which is the root edge of an ear of $G$, and let $G_1$ be a subgraph of $G$ induced by $E_1$.
  We first suppose that $E(G)=E(F)\cup E_1$.
  By the definition of $E_1$, $G_1$ is either a cycle or a disjoint union of paths.
  Suppose that $G_1$ is a cycle $v_1v_2v_3\cdots v_{\ell}v_1$, where $H_i$ is an ear with the root edge $v_iv_{i+1}$ for each $i\in [\ell]$ ($v_{\ell+1}=v_1$) and $x\in V(H_\ell)$.
  Then a subgraph induced by $\bigcup_{i=1}^{\ell-1}V(H_i)$ is a good ear-chain for $x$ with the root edge $v_1v_{\ell}$.
  Otherwise, since there is at least two vertices of $G_1$ of degree $1$,
  there is an edge $v_1v_2$ of $G_1$ such that $d_{G_1}(v_1)=1$ and the ear $H_1$ of $G$ with the root edge $v_1v_2$ satisfies $x\notin V(H_1)\setminus\{v_2\}$.
  Since $d_G(v_1)=d_{G_1}(v_1)+2=3$, $H_1$ is a good ear for $x$.

  Now assume that $E(G)\setminus (E(F)\cup E_1)\neq \emptyset$.
  For each edge $uv\in E(G)\setminus (E(F)\cup E_1)$, let $H_{uv}$ be the subgraph of $G$ induced by the vertices of an $uv$-subpath of $F$ such that $x\notin V(H_{uv})\setminus\{u,v\}$.
  We take an edge $uv\in E(G)\setminus (E(F)\cup E_1)$ so that the order of $H_{uv}$ is as small as possible.
  By the choice of $uv$, we have $E(H_{uv})\subseteq E(F)\cup E_1\cup \{uv\}$.
  Furthermore, we have $E_1\cap E(H_{uv})\neq \emptyset$ since $uv\notin E_1$.
  Let $G_2$ be a subgraph of $G$ induced by $E_1\cap E(H_{uv})$.
  As $G_2$ is a subgraph of $G_1$, which has the maximum degree at most 2, $G_2$ is either a $uv$-path or a disjoint union of paths.
  If $G_2$ is a $uv$-path, then $H_{uv}$ is an ear-chain of $G$ which is good for $x$.
  Otherwise, there is an edge $v_1'v_2'$ of $G_2$ such that $d_{G_2}(v_1')=1$ and $v_1'\notin \{u,v\}$, and let $H_1'$ be an ear of $G$ with the root edge $v_1'v_2'$.
  Since $d_G(v_1')=d_{G_2}(v_1')+2=3$, $H_1'$ is a good ear of $G$ which is good for $x$.
\end{proof}

The following lemmas are used in our proof of Theorem~\ref{thm:outerplanar}.

\begin{lemma}\label{lem:outerplanar ear}
  Let $G$ be a $2$-connected outerplanar graph that is not a cycle.
  Let $H=u_1u_2\cdots u_ru_1$ be an ear of $G$ with the root edge $u_1u_r$, and let $G'=G-(V(H)\setminus\{u_1,u_r\})$.
  Let $L$ be a list assignment of $G$ such that $|L(u_i)|\geq 4$ for every $i\in \{2,3,4,\dots , r-1\}$.
  Suppose that there is a proper $L$-coloring $\phi$ of $G'$ such that either $\mathcal{U}_{\phi}(u_1,G')\neq \emptyset$ or $N_{G'}(u_1)$ is colored with one color.
  Then $\phi$ can be extended to a proper $L$-coloring of $G'\cup H$ such that $\mathcal{U}_{\phi}(u_i,G)\neq \emptyset$ for each $i\in [r-1]$.
\end{lemma}

\begin{proof}
  Note that $u_1, u_r\in V(G')$.
  Without loss of generality, we may assume that $\phi(u_1)=1$.
  We may also assume that either $2\in \mathcal{U}_{\phi}(u_1,G')$ or every vertex of $N_{G'}(u_1)$ is colored by a color $2$.
  We define $\phi(u_i)\in L(u_i)$ for each $i\in \{2,3,\dots ,r-1\}$ as follows.
  Let $\phi(u_2)\in L(u_2)\setminus\{1,2\}$, and for $i=3,4,\dots ,r-3$, we sequentially choose $\phi(u_i)\in L(u_i)\setminus\{\phi(u_{i-2}),\phi(u_{i-1})\}$.
  Then we choose a color in $L(u_{r-2})\setminus\{\phi(u_r),\phi(u_{r-4}),\phi(u_{r-3})\}$ as $\phi(u_{r-2})$, and choose a color in $L(u_{r-1})\setminus\{\phi(u_r),\phi(u_{r-3}),\phi(u_{r-2})\}$ as $\phi(u_{r-1})$.
  By the choice of colors, we have $\phi(u_{i-1})\in\mathcal{U}_{\phi}(u_i,G)$ for each $i\in\{2,3,\dots ,r-1\}$.
  Furthermore, we have $2\in \mathcal{U}_{\phi}(u_1,G)$ if $2\in \mathcal{U}_{\phi}(u_1,G')$ and $\phi(u_2)\in\mathcal{U}_{\phi}(u_1,G)$ otherwise.
  Hence $\phi$ is a desired coloring of $G$.
\end{proof}

\begin{lemma}\label{lem:precolored path}
  Let $P=u_0u_1\cdots u_{s}$ be a path of length $s\geq 3$, and let $L$ be a list assignment of $P$ such that $|L(u_i)|\geq 2$ for $i\in\{0,s\}$, $|L(u_i)|\geq 3$ for $i\in\{1,s-1\}$, and $|L(u_i)|\geq 4$ for other vertices.
  Then $P$ is proper conflict-free $L$-colorable.
\end{lemma}

\begin{proof}
  The proof goes by induction on $s$.
  If $s=3$, then we have $|L(u_0)|, |L(u_3)|\geq 2$ and $|L(u_1)|, |L(u_2)|\geq 3$.
  We may assume that $|L(u_0)|=|L(u_3)|=2$ and $|L(u_1)|=|L(u_2)|=3$.
  If there is a color $\alpha\in L(u_0)\cap L(u_3)$, then we let $\phi(u_0)=\phi(u_3)=\alpha$, and choose $\phi(u_1)\in L(u_1)\setminus\{\alpha\}$ and $\phi(u_2)\in L(u_2)\setminus\{\alpha\}$ so that $\phi(u_1)\neq \phi(u_2)$.
  It is easy to verify that $\phi$ is a proper conflict-free coloring of $P$.
  Hence, we may assume that $L(u_0)\cap L(u_3)=\emptyset$.
  Since $|L(u_0)\cup L(u_3)|=4$, there is a color $\beta\in (L(u_0)\cup L(u_3))\setminus L(u_1)$.
  We assume that $\beta\in L(u_0)$. (The case $\beta\in L(u_3)$ goes symmetrically.)
  Let $\phi(u_0)=\beta$ and choose $\phi(u_3)\in L(u_3)$ arbitrarily.
  Then we choose $\phi(u_2)\in L(u_2)\setminus\{\beta, \phi(u_3)\}$ and choose $\phi(u_1)\in L(u_1)\setminus\{\beta,\phi(u_2),\phi(u_3)\}$.
  It is easy to verify that $\phi$ is a proper conflict-free $L$-coloring of $P$.

  Next we suppose that $s\geq 4$ and the statement holds for paths of length less than $s$.
  Set $\alpha\in L(u_s)$ arbitrarily.
  We define a list assignment $L'$ of $P':=u_0u_1\cdots u_{s-1}$ by $L'(u_i)=L(u_i)\setminus\{\alpha\}$ for $i\in\{s-2, s-1\}$ and $L'(u_i)=L(u_i)$ for $0\leq i\leq s-3$.
  By the induction hypothesis of the claim, $P'$ admits a proper conflict-free $L'$-coloring $\phi$.
  By setting $\phi(u_s)=\alpha$, we can extend $\phi$ to a proper conflict-free $L$-coloring of $P$.
\end{proof}

\section{Proof of Theorem~\ref{thm:outerplanar}}\label{sec:proof}

\setcounter{case}{0}

The proof goes by induction on the number of vertices.

Let $G$ be a connected outerplanar graph that is not isomorphic to $C_5$.
The statement is trivial when $|V(G)|\leq 3$, thus we assume that $|V(G)|\geq 4$.
Let $L$ be a list assignment of $G$ satisfying $|L(v)|=d_G(v)+2$ for every vertex $v\in V(G)$.
We fix an end block $B$ of $G$.
Note that $B=G$ when $G$ is $2$-connected.
Let $x$ be a cut-vertex of $G$ uniquely contained in $B$ when $G$ has at least two blocks, and let $x$ be an arbitrary vertex of $B$ when $B=G$.
We first consider the case when $B$ is isomorphic to the complete graph $K_2$.

\begin{case}\label{case:k2}
    $B$ is isomorphic to $K_2$.
\end{case}

\noindent
Let $v$ be the vertex of $B$ other than $x$.
Note that $d_G(x)\geq 2$.
Let $G'=G-v$.
Since $G'$ is a connected outerplanar graph and $|L(x)|\geq d_G(x)+2=d_{G'}(x)+3$, by the induction hypothesis and Proposition~\ref{prop:c5}, $G'$ admits a proper conflict-free $L$-coloring $\phi$.
As $d_{G'}(x)=d_G(x)-1\geq 1$, let $\alpha$ be a color in $\mathcal{U}_{\phi}(x,G')$.
Then we extend $\phi$ to a proper conflict-free $L$-coloring of $G$ by setting $\phi(v)\in L(v)\setminus \{\phi(x),\alpha\}$.

\medskip
Next, we consider the case where $B$ is a cycle.

\begin{case}\label{case:cycle}
  $B$ is a cycle.
\end{case}

\noindent
When $B=G$, by Proposition~\ref{prop:cycle}, $G$ is proper conflict-free $L$-colorable.
Hence, we assume that $V(G)\setminus V(B)\neq\emptyset$.
Let $G'=G-(V(B)\setminus\{x\})$.
Since $G'$ is a connected outerplanar graph and $|L(x)|\geq d_G(x)+2=d_{G'}(x)+4$, by the induction hypothesis and Proposition~\ref{prop:c5}, $G'$ is proper conflict-free $L$-colorable.
Let $\phi$ be a proper conflict-free $L$-coloring of $G'$.
Since $N_{G'}(x)\neq\emptyset$, without loss of generality, we may assume that $\phi(x)=1$ and there exists a color $\alpha\in \mathcal{U}_{\phi}(x,G')$.

Let $B=xu_0u_1\cdots u_{\ell-2}x$, where $\ell\geq 3$ is the length of $B$.
Note that $|L(u_i)|=d_G(u_i)+2=4$ for every $i\in \{0,1,\dots ,\ell-2\}$.
If $\ell=3$, then we let $\phi(u_0)\in L(u_0)\setminus\{1,\alpha\}$ and let $\phi(u_1)\in L(u_1)\setminus\{1,\alpha,\phi(u_0)\}$ to obtain a proper conflict-free $L$-coloring of $G$.
If $\ell=4$, then we let $\phi(u_0)\in L(u_0)\setminus\{1,\alpha\}$, let $\phi(u_2)\in L(u_2)\setminus\{1,\alpha, \phi(u_0)\}$, and let $\phi(u_1)\in L(u_1)\setminus\{1,\phi(u_0),\phi(u_2)\}$ to obtain a proper conflict-free $L$-coloring of $G$.
Hence, we assume that $\ell\geq 5$.
Let $L_B$ be a list assignment of $B-x$ such that $L_B(u_i)=L(u_i)\setminus\{1,\alpha\}$ for $i\in\{0,\ell-2\}$, $L_B(u_i)=L(u_i)\setminus\{1\}$ for $i\in\{1,\ell-3\}$, and let $L_B(u_i)=L(u_i)$ for $2\leq i\leq \ell-2$.
Then we have $|L_B(u_i)|\geq |L(u_i)|-2=2$ for $i\in\{0,\ell-2\}$, $|L_B(u_i)|\geq |L(u_i)|-1=3$ for $i\in\{1,\ell-3\}$, and $|L_B(u_i)|=|L(u_i)|=4$ for $2\leq i\leq \ell-4$.
By Lemma~\ref{lem:precolored path}, the path $B-x$ admits a proper conflict-free $L_B$-coloring $\phi_B$.
By setting $\phi(u_i)=\phi_B(u_i)$ for every $i\in\{0,1,\dots , \ell-2\}$, we obtain a proper conflict-free $L$-coloring of $G$.

\medskip
By the above two cases, we may assume that $B$ is a 2-connected outerplanar graph that is not a cycle.
According to Lemma~\ref{lem:outerplanar unavoidable}, $B$ has either an ear or an ear-chain that is good for $x$.
  
\begin{case}\label{case:good ear}
  $B$ contains an ear that is good for $x$.
\end{case}

\noindent
Let $H=u_1u_2\cdots u_ru_1$ be a good ear for $x$ with the root edge $u_1u_r$ such that $d_G(u_i)=2$ for each $i\in \{2,3,\dots ,r-1\}$, $x\notin V(H)\setminus\{u_1\}$, and $d_G(u_r)=3$.
Let $G'=G-(V(H)\setminus\{u_1,u_r\})$.
Then $G'$ is a connected outerplanar graph.
Since $|L(u_1)|=d_G(u_1)+2=d_{G'}(u_1)+3$, by the induction hypothesis and Proposition~\ref{prop:c5}, $G'$ admits a proper conflict-free $L$-coloring $\phi$.
By Lemma~\ref{lem:outerplanar ear}, $\phi$ can be extended to a proper $L$-coloring of $G$ such that $\mathcal{U}_{\phi}(v,G)\neq\emptyset$ for every vertex $v$ other than $u_r$.
Since $d_{G'}(u_r)=2$ and $\mathcal{U}_{\phi}(u_r,G')\neq\emptyset$, we have $|\mathcal{U}_{\phi}(u_r,G')|=2$.
Thus, it follows that $|\mathcal{U}_{\phi}(u_r,G)|\geq |\mathcal{U}_{\phi}(u_r,G')|-1\geq 1$, which implies that $\phi$ is a proper conflict-free $L$-coloring of $G$.

\medskip
Before going to the case that $G$ contains an ear-chain, we consider the case when $G$ contains an ear with at least $6$ vertices.

\begin{case}\label{case:long ear}
  $B$ contains an ear $H$ such that $x\notin V(H)$ and $|V(H)|\geq 6$.
\end{case}

\noindent
Let $H=u_1u_2\cdots u_ru_1$ be an ear of $B$ with the root edge $u_1u_r$ such that $d_G(u_i)=2$ for each $i\in \{2,3,\dots ,r-1\}$.
Suppose that $r\geq 6$.
Let $G'=G-\{v_i\mid 2\leq i\leq r-1\}$.
Then $G'$ is a connected outerplanar graph.
Since $|L(u_1)|=d_G(u_1)+2=d_{G'}(u_1)+3$, by the induction hypothesis and Proposition~\ref{prop:c5}, $G'$ admits a proper conflict-free $L$-coloring $\phi$.
Without loss of generality, we may assume $\phi(v_1)=1$ and $\phi(v_r)=2$.
Let $\alpha$ be a color in $\mathcal{U}_{\phi}(u_1,G')$, and let $\beta$ be a color in $\mathcal{U}_{\phi}(u_r,G')$.
Note that $r-2>3$ by the assumption of the case.
Let $\phi(u_2)\in L(u_2)\setminus\{1,\alpha\}$.
We consider the following three subcases.

\begin{subcase}
  $|L(u_{r-1})\cap \{2,\beta\}|\leq 1$.
\end{subcase}

\noindent
For $i=3,4,\dots ,r-3$, we sequentially choose a color in $L(u_i)\setminus\{\phi(u_{i-2}),\phi(u_{i-1})\}$ as $\phi(u_i)$.
Then we choose a color in $L(u_{r-2})\setminus\{2,\phi(u_{r-4}),\phi(u_{r-3})\}$ as $\phi(u_{r-2})$.
Since $|L(u_{r-1})\cap \{2,\beta\}|\leq 1$, we can choose $\phi(u_{r-1})\in L(u_{r-1})\setminus\{2,\beta,\phi(u_{r-3}),\phi(u_{r-2})\}$.
It is easy to verify that $\phi$ is a proper conflict-free $L$-coloring of $G$.

\begin{subcase}
  $\{2,\beta\}\subseteq L(u_{r-1})$ and $L(u_{r-2})\setminus L(u_{r-1})\neq \emptyset$.
\end{subcase}

\noindent
Let $\phi(u_{r-2})\in L(u_{r-2})\setminus L(u_{r-1})$. 
Note that $\phi(u_{r-2})\neq 2$ since $2\in L(u_{r-1})$.
For $i=3,4,\dots , r-3, r-1$, we sequentially choose a color $\phi(u_i)\in L(u_i)$ as follows.
For $i\in \{3,4,\dots ,r-5\}$, we let $\phi(u_i)\in L(u_i)\setminus\{\phi(u_{i-2}),\phi(u_{i-1})\}$.
For $i\in \{r-4,r-3\}$, we let $\phi(u_i)\in L(u_i)\setminus\{\phi(u_{r-2}),\phi(u_{i-2}),\phi(u_{i-1})\}$.
Since $\phi(u_{r-2})\notin L(u_{r-1})$, we may choose $\phi(u_{r-1})\in L(u_{r-1})\setminus\{2,\beta,\phi(u_{r-3}),\phi(u_{r-2})\}$.
It is easy to verify that $\phi$ is a proper conflict-free $L$-coloring of $G$.

\begin{subcase}
  $\{2,\beta\}\subseteq L(u_{r-1})$ and $L(u_{r-2})\setminus L(u_{r-1})= \emptyset$.
\end{subcase}

\noindent
Since $|L(u_{r-1})|=|L(u_{r-2})|=4$, the subcase assumption implies that $L(u_{r-1})=L(u_{r-2})$, and thus $\beta\in L(u_{r-2})$.
Let $\phi(u_{r-2})=\beta$.
For $i=3,4,\dots , r-3, r-1$, we sequentially choose a color $\phi(u_i)\in L(u_i)$ as follows.
For $i\in \{3,4,\dots ,r-5\}$, we let $\phi(u_i)\in L(u_i)\setminus\{\phi(u_{i-2}),\phi(u_{i-1})\}$.
For $i\in \{r-4,r-3\}$, we let $\phi(u_i)\in L(u_i)\setminus\{\beta,\phi(u_{i-2}),\phi(u_{i-1})\}$.
Then, we choose a color in $L(u_{r-1})\setminus\{2,\beta,\phi(u_{r-3})\}$ as $\phi(u_{r-1})$.
It is easy to verify that $\phi$ is a proper conflict-free $L$-coloring of $G$.
This completes the proof for Case~\ref{case:long ear}.

\medskip
Now we consider the case that $B$ contains an ear-chain that is good for $x$.

\begin{case}\label{case:ear-chain}
  $B$ contains an ear-chain that is good for $x$.
\end{case}

\noindent
Let $H$ be an ear-chain of $B$ consisting of ears $H_1,H_2,\dots,H_{s-1}$ $(s\geq 3)$ of $B$ such that the root edge of $H_i$ is $v_iv_{i+1}$ for each $i\in [s-1]$.
Without loss of generality, we may assume that $x\notin V(H)\setminus\{v_1\}$.
By Case~\ref{case:long ear}, we may assume that $|V(H_{s-1})|\leq 5$.
Note that $|L(v_i)|=6$ for each $i\in \{2,3,\dots ,s-1\}$ and $|L(v)|=4$ for every vertex $v\in V(H)\setminus\{v_i\mid 1\leq i\leq s\}$.
We consider the following four subcases.
In each of the following cases, let $G'=G-(V(H)\setminus\{v_1,v_s\})$, which is a connected outerplanar graph.

\begin{subcase}\label{subcase:ear-chain more than 4}
  $s\geq 4$.
\end{subcase}

\noindent
Since $|L(v_1)|=d_G(v_1)+2=d_{G'}(v_1)+4$, by the induction hypothesis and Proposition~\ref{prop:c5}, $G'$ admits a proper conflict-free $L$-coloring $\phi$.
Without loss of generality, we may assume that $\phi(v_1)=1$ and $\phi(v_s)=2$.
Let $\alpha$ be a color in $\mathcal{U}_{\phi}(v_1,G')$ and let $\beta$ be a color in $\mathcal{U}_{\phi}(v_s,G')$.
Note that $\alpha\neq 1$ and $\beta\neq 2$.
Let $H_{s-1}=v_{s-1}u_2u_3\cdots u_{r-1}v_sv_{s-1}$, where $r\in\{3,4,5\}$.

We give a coloring of $\{v_i\mid 2\leq i\leq s-1\}\cup V(H_{s-1})$ as follows.
We first choose $\phi(u_{r-1})\in L(u_{r-1})\setminus\{2,\beta\}$, and for $j=r-2,r-3,\dots ,2$, we sequentially choose a color in $L(u_j)\setminus\{2,\phi(u_{j+1}),\phi(u_{j+2})\}$ as $\phi(u_j)$, where $u_r=v_s$.
Next, we color $v_i$'s.
Let $\phi(v_{s-1})\in L(v_{s-1})\setminus \{2,\beta,\phi(u_2),\phi(u_3)\}$.
For each $i\in \{2,3,\dots, s-2\}$, since $|L(v_i)|=6$, we can choose $\phi(v_i)\in L(v_i)$ so that $\phi(v_2)\notin \{1,\alpha\}$, $\phi(v_{s-2})\notin \{2,\phi(v_{s-1}),\phi(u_2)\}$, and $\phi(v_i)\neq \phi(v_{i-1})$ for every $i\in \{2,3,\dots , s-2\}$.
Note that $\mathcal{U}_{\phi}(v_1,G)\neq \emptyset$ by the choice of $\phi(v_2)$.

For the remaining ears $H_1, H_2,\dots ,H_{s-2}$, we sequentially apply Lemma~\ref{lem:outerplanar ear} to extend $\phi$ to a proper $L$-coloring of $G$ such that $\mathcal{U}_{\phi}(v,G)\neq\emptyset$ for every vertex $v\in V(H_i)\setminus\{v_{i+1}\}$.
Indeed, for $i\in \{2,3,\dots , s-2\}$, when we color $H_i$ with the root edge $v_iv_{i+1}$, the vertex $v_i$ has three colored neighbors.
Thus, either $\mathcal{U}_{\phi}(v_i,G)\neq \emptyset$ or all three neighbors receive the same color, so we can apply Lemma~\ref{lem:outerplanar ear} to color $H_i$.
By the choice of colors, it is easy to verify that $\mathcal{U}_{\phi}(v,G)\neq \emptyset$ for every vertex $v\in V(G)\setminus\{v_{s-1}\}$.
Since $d_G(v_{s-1})=4$ and $\phi(v_{s-2})$, $\phi(v_s)$ and $\phi(u_2)$ are three distinct colors, we have $\mathcal{U}_{\phi}(v_{s-1},G)\neq \emptyset$.
Thus, $\phi$ is a proper conflict-free $L$-coloring of $G$.

\begin{subcase}\label{subcase:ear-chain 3}
  $s=3$ and $|V(H_2)|=3$.
\end{subcase}

\noindent
Let $H_2=v_2u_2v_3v_2$.
Since $|L(v_1)|=d_G(v_1)+2=d_{G'}(v_1)+4$, by the induction hypothesis and Proposition~\ref{prop:c5}, $G'$ admits a proper conflict-free $L$-coloring $\phi$.
Without loss of generality, we may assume that $\phi(v_1)=1$ and $\phi(v_3)=2$.
Let $\alpha$ be a color in $\mathcal{U}_{\phi}(v_1,G')$ and $\beta$ be a color in $\mathcal{U}_{\phi}(v_3,G')$.
We choose $\phi(u_2)\in L(u_2)\setminus\{1,2,\beta\}$, and choose $\phi(v_2)\in L(v_2)\setminus\{1,2,\alpha,\beta,\phi(u_2)\}$.
Then we apply Lemma~\ref{lem:outerplanar ear} to the ear $H_1$ to extend $\phi$ to a proper $L$-coloring of $G$ such that $\mathcal{U}_{\phi}(v,G)\neq\emptyset$ for every vertex $v\in V(H_1)\setminus\{v_2\}$.
By the choice of colors, it is easy to verify that $\mathcal{U}_{\phi}(v,G)\neq \emptyset$ for every vertex $v\in V(G)\setminus\{v_2\}$.
Since $d_G(v_2)=4$ and $\phi(v_1)$, $\phi(v_3)$ and $\phi(u_2)$ are three distinct colors, we have $\mathcal{U}_{\phi}(v_2,G)\neq \emptyset$, and thus $\phi$ is a proper conflict-free $L$-coloring of $G$.

\begin{subcase}\label{subcase:ear-chain 4}
  $s=3$ and $|V(H_2)|=4$.
\end{subcase}

\noindent
Let $H_2=v_2u_2u_3v_3v_2$.
We fix a color $\gamma$ such that $\gamma\in L(u_2)\setminus L(u_3)$ if $L(u_2)\setminus L(u_3)\neq \emptyset$ and $\gamma\in L(u_2)$ otherwise.
Let $L'$ be a list assignment of $G'$ such that $L'(v_1)=L(v_1)\setminus \{\gamma\}$, $L'(v_3)=L(v_3)\setminus \{\gamma\}$, and $L'(v)=L(v)$ for every vertex $v\in V(G')\setminus\{v_1,v_3\}$.
Since $|L'(v)|\geq d_{G'}(v)+2$ for every vertex $v\in V(G')$ and $|L'(v_1)|\geq |L(v_1)|-1 = d_G(v_1)+1 = d_{G'}(v_1)+3$, by the induction hypothesis and Proposition~\ref{prop:c5}, $G'$ admits a proper conflict-free $L'$-coloring $\phi$.
Without loss of generality, we may assume that $\phi(v_1)=1$ and $\phi(v_3)=2$.
Let $\alpha$ be a color in $\mathcal{U}_{\phi}(v_1,G')$ and let $\beta$ be a color in $\mathcal{U}_{\phi}(v_3,G')$.
Note that $\gamma\notin \{1,2\}$ by the definition of $L'$.

Suppose first that $\beta=1$.
We let $\phi(u_2)\in L(u_2)\setminus\{1,2\}$, $\phi(u_3)\in L(u_3)\setminus\{1,2,\phi(u_2)\}$, and $\phi(v_2)\in L(v_2)\setminus\{1,2,\alpha,\phi(u_2),\phi(u_3)\}$.
Then we apply Lemma~\ref{lem:outerplanar ear} to the ear $H_1$ to extend $\phi$ to a proper $L$-coloring of $G$ such that $\mathcal{U}_{\phi}(v,G)\neq \emptyset$ for every vertex $v\in V(G)\setminus\{v_2\}$.
Since $d_G(v_2)=4$ and $\phi(v_1)$, $\phi(v_3)$, $\phi(u_2)$ are three distinct colors, we infer that $\mathcal{U}_{\phi}(v_2,G)\neq \emptyset$, and thus $\phi$ is a proper conflict-free $L$-coloring of $G$.
  
Thus, we may assume that $\beta\neq 1$.
Then there is a color $\gamma'\in L(u_2)\setminus\{1,2\}$ (possibly $\gamma'=\gamma$) such that $|L(u_3)\setminus\{2,\beta,\gamma'\}|\geq 2$.
Indeed, if $L(u_2)\setminus L(u_3)\neq \emptyset$, then we know $\gamma\in L(u_2)\setminus\{1,2\}$ and $|L(u_3)\setminus\{2,\beta,\gamma\}|=|L(u_3)\setminus\{2,\beta\}|\geq 2$ by the choice of $\gamma$.
Otherwise, since $L(u_2)=L(u_3)$ and $\{1,2\}\neq \{2,\beta\}$, either $|L(u_3)\setminus\{2,\beta\}|\geq 3$, or $|L(u_3)\setminus\{2,\beta\}|=2$ and $L(u_2)\setminus\{1,2\}\neq L(u_3)\setminus\{2,\beta\}$.
In the former case, we have $|L(u_3)\setminus\{2,\beta,\gamma\}|\geq |L(u_3)\setminus\{2,\beta\}|-1\geq 2$.
In the latter case, there is a color $\gamma'\in (L(u_2)\setminus\{1,2\})\setminus(L(u_3)\setminus\{2,\beta\})$, and thus $|L(u_3)\setminus\{2,\beta,\gamma'\}|=|L(u_3)\setminus\{2,\beta\}|\geq 2$.
Let $\phi(u_2)=\gamma'$.
We choose $\phi(v_2)\in L(v_2)\setminus\{1,2,\alpha,\beta,\gamma'\}$.
Since $|L(u_3)\setminus\{2,\beta,\gamma',\phi(v_2)\}|\geq |L(u_3)\setminus\{2,\beta,\gamma'\}|-1\geq 1$, we can choose a color $\phi(u_3)\in L(u_3)\setminus\{2,\beta,\gamma',\phi(v_2)\}$.
By using Lemma~\ref{lem:outerplanar ear} to the ear $H_1$, we extend $\phi$ to a proper $L$-coloring of $G$ such that $\mathcal{U}_{\phi}(v,G)\neq \emptyset$ for every vertex $v\in V(G)\setminus\{v_2\}$.
Since $d_G(v_2)=4$ and $\phi(v_1)$, $\phi(v_3)$, $\phi(u_2)$ are three distinct colors, we again infer that $\mathcal{U}_{\phi}(v_2,G)\neq \emptyset$, and thus $\phi$ is a proper conflict-free $L$-coloring of $G$.

\begin{subcase}\label{subcase:ear-chain 5}
  $s=3$ and $|V(H_2)|=5$.
\end{subcase}

\noindent
Let $H_2=v_2u_2u_3u_4v_3v_2$.
Since $|L(v_1)|=d_G(v_1)+2=d_{G'}(v_1)+4$, by the induction hypothesis and Proposition~\ref{prop:c5}, $G'$ admits a proper conflict-free $L$-coloring $\phi$.
Without loss of generality, we may assume that $\phi(v_1)=1$ and $\phi(v_3)=2$.
Let $\alpha$ be a color in $\mathcal{U}_{\phi}(v_1,G')$ and $\beta$ be a color in $\mathcal{U}_{\phi}(v_3,G')$.

First, we choose $\tilde{L}(v_2)\subseteq L(v_2)\setminus\{1,2,\alpha,\beta\}$, $\tilde{L}(u_3)\subseteq L(u_3)\setminus\{2\}$ and $\tilde{L}(u_4)\subseteq L(u_4)\setminus\{2,\beta\}$ so that $|\tilde{L}(v_2)|=|\tilde{L}(u_4)|=2$ and $|\tilde{L}(u_3)|=3$.
Next, choose a color $\phi(v_2)\in \tilde{L}(v_2)$ and $\phi(u_4)\in \tilde{L}(u_4)$ as follows.
If $\tilde{L}(v_2)\cap \tilde{L}(u_4)\neq \emptyset$, then we let $\phi(v_2)=\phi(u_4)$ be a color in $\tilde{L}(v_2)\cap \tilde{L}(u_4)$.
Otherwise, since $|\tilde{L}(u_3)|=3$ and $|\tilde{L}(v_2)\cup \tilde{L}(u_4)|=4$, we choose $\phi(v_2)\in \tilde{L}(v_2)$ and $\phi(u_4)\in \tilde{L}(u_4)$ so that $\phi(v_2)$ or $\phi(u_4)$ does not belong to $\tilde{L}(u_3)$.
In both cases, we have $|\tilde{L}(u_3)\setminus\{\phi(v_2),\phi(u_4)\}|\geq 2$.

Now we apply Lemma~\ref{lem:outerplanar ear} to the ear $H_1$ to extend $\phi$ to a proper $L$-coloring of $G-\{u_2,u_3\}$ so that $\mathcal{U}_{\phi}(v,G)\neq \emptyset$ for every vertex $v\in V(G)\setminus\{v_2, u_2, u_3\}$.
Since $v_2$ has three colored neighbors and $\phi(v_1)\neq \phi(v_3)$, there is a color $\gamma\in \mathcal{U}_{\phi}(v_2,G)$.
We choose $\phi(u_2)\in L(u_2)\setminus\{\phi(v_2),\phi(u_4),\gamma\}$.
Since $|\tilde{L}(u_3)\setminus\{\phi(v_2),\phi(u_4)\}|\geq 2$, we can choose $\phi(u_3)\in \tilde{L}(u_3)\setminus\{\phi(v_2),\phi(u_2),\phi(u_4)\}$.
It is easy to verify that $\phi$ is a proper conflict-free $L$-coloring of $G$.

\medskip
This completes the proof of Case~\ref{case:ear-chain} and the proof of Theorem~\ref{thm:outerplanar}.

\section{Concluding remarks and futher work}\label{sec:conclusion}

In this paper, we show that every connected outerplanar graph other than the $5$-cycle is proper conflict-free $({\rm degree}+2)$-choosable.
There are infinitely many connected outerplanar graphs that are not proper conflict-free $({\rm degree}+1)$-choosable, and hence we cannot reduce $({\rm degree}+2)$ to $({\rm degree}+1)$.

Since the graph $G$ constructed in the proof of Proposition~\ref{prop:degree+1} has a cut-vertex $v$, a possible question is whether there is a $2$-connected graph that is not proper conflict-free $({\rm degree}+1)$-choosable.
It was shown in Caro, Petru\v{s}evski, and \v{S}krekovski~\cite{CPS2023} that the cycle of length $\ell$ is not proper conflict-free $3$-colorable if $\ell$ is not divisible by $3$.
Hence, cycles of length $\ell\not\equiv 0 \pmod{3}$ are examples of $2$-connected graphs that are not proper conflict-free $({\rm degree}+1)$-choosable.

Other examples are theta graphs with specified parameters.
For positive integers $a$, $b$ and $c$, \emph{Theta graph} $\Theta_{a,b,c}$ is a graph consisting of two vertices of degree $3$ and three internally disjoint paths of length $a$, $b$ and $c$ joining them.
Then the following holds.

\begin{proposition}\label{prop:theta graph}
    Let $\ell_1\geq 4$ and $\ell_2\geq 4$ be integers such that $\ell_1\equiv 1\pmod{3}$ and $\ell_2\equiv 1\pmod{3}$.
    Then, Theta graph $\Theta_{1,\ell_1,\ell_2}$ is not proper conflict-free $({\rm degree}+1)$-choosable.
\end{proposition}

\begin{proof}
    Let $V(G)=\{v_1,v_2\}\cup \{u_i\mid i\in [\ell_1-1]\}\cup \{w_j\mid j\in [\ell_2-1]\}$ such that $v_1v_2$, $v_1u_1u_2\cdots u_{\ell_1-1}v_2$, and $v_1w_1w_2\cdots w_{\ell_2-1}v_2$ respectively form paths of length $1$, $\ell_1$, and $\ell_2$ of $G$.
    Note that $d_G(v_1)=d_G(v_2)=3$ and all $u_i$'s and $w_j$'s are of degree $2$.
    Let $L$ be a list assignment of $G$ such that $L(v_1)=L(v_2)=\{1,2,3,4\}$ and $L(u_i)=L(w_j)=\{1,2,3\}$ for every $i\in [\ell_1-1]$ and every $j\in [\ell_2-1]$.
    We show that $G$ is not proper conflict-free $L$-colorable.

    Suppose that $G$ admits a proper conflict-free $L$-coloring $\phi$.
    Since $v_1v_2\in E(G)$, without loss of generality, we may assume that $\phi(v_1)=1$.
    Let $u_0=v_1$.
    Since $\phi(u_i)\in L(u_i)\setminus\{\phi(u_{i-2}),\phi(u_{i-1})\}$ for every $i\in \{2,3,\dots , \ell_1-1\}$, it follows that $\phi(u_i)=\phi(u_{i'})$ for every $i,i'\in \{0,1,\dots , \ell_1-1\}$ with $i\equiv i'\pmod{3}$.
    In particular, since $\ell_1\equiv 1\pmod{3}$, we infer that $\phi(u_{\ell_1-1})=\phi(u_0)=1$.
    Similarly, we have $\phi(w_{\ell_2-1})=1$.
    Then the vertex $v_2$ has three neighbors $v_1$, $u_{\ell_1-1}$, and $w_{\ell_2-1}$ with color $1$, a contradiction.
\end{proof}

The only 2-connected graphs we know that are not proper conflict-free $({\rm degree}+1)$-choosable are the above two examples; that is, cycles of length $\ell\not\equiv 0 \pmod{3}$ and theta graphs $\Theta_{1,\ell_1,\ell_2}$ with $\ell_1\equiv\ell_2\equiv 1 \pmod{3}$.
So one may consider the following problem for further research.

\begin{problem}\label{problem:2conn}
    Find some other (possibly all) 2-connected graphs that are not proper conflict-free $({\rm degree}+1)$-choosable.
\end{problem}

In our constructions of graphs that are not proper conflict-free $({\rm degree}+1)$-choosable, vertices of degree $2$ play important roles.
We wonder whether things change by excluding them.

\begin{question}\label{question:min degree 3}
    Is every graph with minimum degree at least $3$ proper conflict-free $({\rm degree}+1)$-choosable?
\end{question}

\section*{Acknowledgement}

We appreciate the referees for their careful reading of the manuscript and pointing out a fault in a lemma in the previous version.

M. Kashima has been supported by JSPS KAKENHI Grant No. 25KJ2077.
R. \v{S}krekovski has been partially supported by the Slovenian Research Agency and Innovation (ARIS) program P1-0383, project J1-3002, and the annual work program of Rudolfovo. 
R. Xu has been supported by National Science Foundation for Young Scientists of China under Grant No. 12401472, and Zhejiang Provincial Natural Science Foundation of
China under Grant No. LQN25A010011.

\end{document}